\documentclass[11pt]{amsart}
\usepackage{amsmath,amssymb}

\usepackage{latexsym}
\usepackage[dvips]{graphics}
\usepackage[latin1]{inputenc}
\usepackage{amsthm}
\usepackage{amsfonts}
\usepackage{comment}

\usepackage{graphicx,xy,pstricks}
\input xy
\xyoption{all}

\newtheorem{theorem}{Theorem}[section]
\newtheorem{proposition}{Proposition}[section]
\newtheorem{lemma}{Lemma}[section]

\theoremstyle{remark}
\newtheorem{remark}{Remark}[section]
\newtheorem{definition}{Definition}[section]

\def \Z{\mathbb Z}
\def \R{\mathbb R}
\def \C{\mathbb C}

\def \M{\rm Mod}

\theoremstyle{plain}      

\renewcommand{\d}{\mathrm{d}}
\newcommand{\g}{\mathfrak{g}}
\renewcommand{\hom}{\operatorname{Hom}}
\renewcommand{\phi}{\varphi}

\newcommand{\mo}{\mathcal{M}}
\newcommand{\tr}{\operatorname{Tr}}
\newcommand{\hol}{\operatorname{Hol}}

\newcommand{\ad}{\operatorname{Ad}}
\renewcommand{\sp}{\operatorname{Sp}}
\newcommand{\rep}{\mathcal{R}}
\newcommand{\su}{\mathrm{SU}_2}
\renewcommand{\span}{\operatorname{Span}}

\title{The first Johnson subgroups act ergodically on $\su$-character varieties}

\author{Louis Funar}
\address{Institut Fourier BP 74, UMR 5582,  Universit\'e Grenoble I, 38402
Saint-Martin-d'H\`eres Cedex, France }
\email{funar@fourier.ujf-grenoble.fr}
\author{Julien March\'e}
\address{
Centre de Math\'ematiques Laurent Schwartz,
\'Ecole Polytechnique
Route de Saclay, 91128 Palaiseau Cedex,
France
}
\email{marche@math.polytechnique.fr}
\begin{document}
\maketitle
\begin{abstract}
We show that the first Johnson subgroup of the mapping class group 
of a surface $\Sigma$ of genus greater than one acts ergodically 
on the moduli space of representations of $\pi_1(\Sigma)$ in $\su$. Our 
proof relies on a local description of the latter space around the 
trivial representation and on the Taylor expansion of trace functions. 

\vspace{0.2cm}

\end{abstract}

\section{Introduction and statements} 
Let $\Sigma$ be a compact  orientable surface 
and $G$ a compact semi-simple Lie group.
The space of  all homomorphisms  
${\rm Hom}(\pi_1(\Sigma),G)$ will be denoted 
$\rep(\Sigma,G)$. We will consider then the quotient 
$\mo(\Sigma, G)={\mathcal R}(\Sigma, G)/G$   by the 
conjugacy  $G$-action. 
Let $\M(\Sigma)$ denote the mapping class group of the surface 
$\Sigma$, namely the group of isotopy classes of homeomorphisms 
preserving the orientation of $\Sigma$ fixing the boundary. 
The group ${\rm Aut}^+(\pi_1(\Sigma))$ of automorphisms of $\pi_1(\Sigma)$ 
preserving the orientation acts on ${\mathcal R}(\Sigma,G)$ 
by left composition and induces an action 
of the mapping class group $\M(\Sigma)$ on $\mo(\Sigma,G)$.

Recall that, if $\Sigma$ is closed and orientable then 
the dense and open subset of non-singular points of $\mo(\Sigma,\su)$ 
has a $\M(\Sigma)$-invariant symplectic structure, 
which was defined by Goldman in \cite{Go1}. 
This induces a volume form  on the non-singular part of $\mo(\Sigma,\su)$  
and thus a  $\M(\Sigma)$-invariant measure on $\mo(\Sigma,\su)$. 

The main purpose of this paper is to understand the dynamical properties 
of the $\M(\Sigma)$  action, with respect to this invariant measure. 
The first result in this direction is due to Goldman (see \cite{Go2}) 
who proved that $\M(\Sigma)$ acts ergodically on $\mo(\Sigma,\su)$. 
This ergodicity statement was further extended to all compact connected 
Lie groups by Pickrell and Xia (\cite{PX}). 

\begin{definition}
The first Johnson subgroup $\mathcal K(\Sigma)$ is the subgroup of $\M(\Sigma)$ 
generated by the Dehn twists along separating simple curves 
on $\Sigma$.  
\end{definition}

Johnson proved (see \cite{J}) that $\mathcal K(\Sigma)$ 
is a normal subgroup of infinite index of the 
Torelli subgroup of $\M(\Sigma)$, which is the subgroup 
of mapping classes of homeomorphisms acting trivially on the 
surface homology.

\begin{theorem}\label{main}
Let $\Sigma$ be a closed orientable surface of genus $g\geq 2$. 
Then the Johnson subgroup $\mathcal K(\Sigma)$ acts ergodically on 
$\mo(\Sigma,\su)$. 
\end{theorem}
This solves affirmatively Conjecture 1.8 of Goldman from \cite{Go2}. 
Previously the conjecture has been proved only for surfaces of genus $g=1$ with two  
boundary components (for almost all boundary monodromies) following 
a different approach,  by Goldman and Xia (see \cite{GX2}). 

\begin{remark}
It is sufficient to show that the lift of $\mathcal K(\Sigma)$ in ${\rm Aut}^+(\pi_1(\Sigma))$  acts ergodically on 
$\rep(\Sigma,\su)$,  for a closed orientable surface $\Sigma$ of genus 
$g\geq 2$. 
\end{remark}

The proof can be extended with only minor modifications 
to the case where $G=\su\times \su\times\cdots \times \su$ is the 
direct product of $k$ factors. 
Therefore we obtain: 

\begin{theorem}\label{weakmixing}
Let $\Sigma$ be a closed orientable surface of genus $g\geq 2$. 
Then the Johnson subgroup $\mathcal K(\Sigma)$ acts ergodically on 
$\mo(\Sigma,\su)\times \mo(\Sigma,\su)\times \cdots\times \mo(\Sigma,\su)$,
by means of the diagonal action. In particular, the action of 
$\mathcal K(\Sigma)$ on $\mo(\Sigma,\su)$ is 
weakly mixing.  
\end{theorem}

\begin{remark}
Although the result is stated here for $G=\su$, 
the proof could be adapted to simply connected compact groups. 
Observe also that the action of $\mathcal K(\Sigma)$ on 
$\mo(\Sigma,\mathrm{U}_1)$ is trivial, contrary to the action of 
$\M(\Sigma)$, which is known to be ergodic. 
\end{remark}

\begin{remark}
One can ask if a similar result holds 
true when $\Sigma$ is a compact surface with boundary 
$\partial \Sigma$. In this case we consider the action of 
$\mathcal K(\Sigma)$ on  the symplectic leaves  
$\mo(\Sigma, \su; (\alpha_i)_{i\in H_0(\partial\Sigma)})$ of 
$\mo(\Sigma,\su)$, consisting of those classes of representations 
whose conjugacy classes on each circle of $\partial \Sigma$ 
are the fixed $\alpha_i$. We have to require that the genus $g\geq 2$ or 
$g=1$ and the number of boundary components is at least $3$.  
Our proof does not seem to extend to this case as we need that the 
trivial representation belongs to the space of representations considered. 
\end{remark}

\begin{remark}
The present proof heavily uses the symplectic structure on 
$\mo(\Sigma,\su)$.  
Notice that when $\Sigma$ is closed but non-orientable 
the space $\mo(\Sigma,\su)$ still admits an invariant 
volume form (actually it is enough to have an invariant measure class)  
and the action of $\M(\Sigma)$ was proved to be 
ergodic in \cite{Pal} for all but a few explicit small surfaces.  
However the present proof does not seem to extend to the
non-orientable case. 
\end{remark}

The proof of Theorem \ref{main} and hence the paper is organized in the 
following way. In the first section, we describe the local structure of the 
representation space around the trivial representation. Then we compute the 
Taylor expansion of trace functions, in particular we show that the first 
non trivial term in the expansion of the trace function of a separating curve 
has order 4.   
In Section 4, we show that these trace functions are generating around 
the trivial representation, in the sense that their differentials generate 
the cotangent space. Then we conclude our proof in Section 5 by an 
argument similar to the one in \cite{GoX}. 

\vspace{0.2cm}
{\bf Acknowledgements:}
The first author was supported 
by the ANR 2011 BS 01 020 01 ModGroup and 
the second author was supported by the ANR-08-JCJC-0114-01. The authors are indebted to W. Goldman for valuable discussions.

 \section{Local structure of representation spaces around the trivial representation}
 We recall the following result due to Goldman and Millson 
(see \cite{gm2}) for general K\"ahler manifolds, which for 
the case of surfaces was already obtained by Arms, Marsden and 
Moncrief \cite{amm}. This case is detailed in the appendix of \cite{gm1}. 
 
 \begin{proposition}
Let $\Sigma$ be a compact Riemann surface and $G$ a compact Lie group. 

Let $P$ be flat principal $G$-bundle over $\Sigma$ and denote by $F(P)$ 
the space of flat connections on $P$. Given a flat connection $A$ on $P$, 
let 
$Z^1(\Sigma,\ad P)$ denote the space of infinitesimal deformations of $A$ 
inside $F(P)$.

Then, there exists an analytical diffeomorphism between a neighborhood
 of $A$ in $F(P)$ and a neighborhood of $0$ in the subset of 
$Z^1(\Sigma,\ad P)$ consisting of $\ad P$-valued 1-forms $\eta$ such 
that $[\eta,\eta]$ is exact. 
 \end{proposition}
 
 We reformulate this result as  it was stated in the corollary 
from (\cite{gm1} p.143), by specializing  to the case where 
$P$ is the trivial bundle and $A$ the trivial connection. 

Given a base point $x$ in $\Sigma$, we denote by $\rep(\Sigma,G)$ 
the variety $\hom(\pi_1(\Sigma,x),G)$ and by 
$\hol:F(P)\to\rep(\Sigma,G)$ the holonomy representation.

\begin{definition}
We denote by $C$ the tangent cone  at the identity representation,   
namely the set of elements $u$ in $H^1(\Sigma,\g)$ satisfying $[u,u]=0$. 
Identifying $H^1(\Sigma,\g)$ with $\hom(H_1(\Sigma,\R),\g)$ we define 
$C_{\rm irr}\subset C$ as the set of surjective maps. 
\end{definition}
 
We have then: 
 
 \begin{proposition}
Let $\Sigma$ be a compact Riemann surface and $G$ a compact Lie group. 
There is an analytic function $F:H^1(\Sigma,\g)\to \Omega^1(\Sigma,\g)$ which satisfies
\begin{enumerate}
\item $F(0)=0$ and $D_0F(u)$ is closed and cohomologous to $u$.  
\item If $[u,u]=0$ then $d F(u)+\frac{1}{2}[F(u),F(u)]=0$. 
\item $F$ maps $C_{\rm irr}$ to flat connections with irreducible monodromy. 
\item The map $\hol\circ F:C\to\rep(\Sigma,G)$ is a real analytic diffeomorphism in the neighborhood of 0.
\item The map $F$ is equivariant with respect to the adjoint action of $G$.
 \end{enumerate}
 \end{proposition}
 
 \begin{remark} The proof of this theorem uses harmonic theory and is 
an application of the implicit function theorem to the Kuranishi map 
in $\Omega^1(\Sigma,\g)$, see \cite{gm1}. The space of 1-forms are 
topologized using Sobolev $s$-norms where $s$ is sufficiently large so 
that the connections are at least $C^2$.
 \end{remark}
 
 \section{Taylor expansion of trace functions}
 Let $\gamma:\R/\Z\to \Sigma$ a parametrized simple curve and 
set $G=\su\subset M_2(\C)$, where $M_2(\C)$ denotes the algebra 
of 2-by-2 matrices with complex entries.
 
 From now on, we will identify $H^2(\Sigma,\g)$ with $\g$ by the evaluation on the fundamental class. We will denote by $\langle\cdot,\cdot\rangle$ the pairing between homology and cohomology. Choose once for all a norm $|\cdot |$ on $H^1(\Sigma,\mathfrak g)$. 
 \begin{proposition}\label{taylor}
 If $\gamma$ is non separating:
$$ \tr \hol_\gamma F(u)=2+\frac{1}{4}\tr\langle\gamma, u\rangle^2+O(|u|^3)$$
If $\gamma$ is separating it decomposes the surface in two parts $\Sigma'\cup_{\gamma} \Sigma''$. Denote by $u=u'+u''$ the decomposition of $u$ in $H^1(\Sigma,\g)=H^1(\Sigma',\g)\oplus H^1(\Sigma'',\g)$. We have
$$ \tr \hol_\gamma F(u)=2+\frac{1}{16}\tr[u'\wedge u']^2+O(|u|^5)$$
  \end{proposition}
\begin{proof}
Given a connexion form $\alpha\in \Omega^1(\Sigma,\g)$, we consider the well-known formula:
$$\hol_\gamma \alpha=\sum_{n\ge 0} \frac{1}{n!}\int_{\Delta^n}(\gamma^*\alpha)^n,$$
where $\int_{\Delta^n}\alpha^n=\int_{0<t_1<\cdots t_n<1}\beta(t_1)\cdots\beta(t_n)\d t_1\cdots \d t_n$ and $\alpha=\beta(t)\d t$.

From the identity $\tr(\hol_\gamma\alpha)^{-1}=\tr(\hol_\gamma\alpha)$, we observe that we can invert the parametrization of $\gamma$ without changing the result. Writing $\beta(t)\d t=\gamma^*\alpha$ we compute:
\begin{eqnarray*}
\tr(\hol_\gamma \alpha)&=&\sum_{n\ge 0} \frac{1}{n!}\tr\int_{t_1<\cdots <t_n}\beta(1-t_1)\cdots\beta(1-t_n)\d t_1\cdots \d t_n.\\
&=&\sum_{n\ge 0} \frac{1}{n!}\tr\int_{t_1<\cdots <t_n}\beta(t_n)\cdots\beta(t_1)\d t_1\cdots \d t_n.\\
&=&\sum_{n\ge 0} \frac{(-1)^n}{n!}\tr\int_{t_1<\cdots <t_n}\beta(t_1)\cdots\beta(t_n)\d t_1\cdots \d t_n.
\end{eqnarray*}
This implies that only the even values of $n$ contribute to the sum. Observe that 
$$\tr\int_{\Delta^2}(\gamma^*\alpha)^2=\frac{1}{2}\tr\int \beta(t_1)\beta(t_2)\d t_1 \d t_2=\frac{1}{2}\tr\left(\int_\gamma\alpha\right)^2.$$

We now apply these formulas to $\alpha=F(u)$. Developing $F(u)$ into Taylor series we can write $F(u)=F_1(u)+F_2(u)+O(|u|^3)$ where $F_1(u)$ and $F_2(u)$ are 1-forms satisfying 
\begin{enumerate}
\item $\d F_1(u)=0$.
\item $F_1(u)$ is cohomologous to $u$.
\item $\d F_2(u)+\frac{1}{2}[F_1(u),F_1(u)]=0$.
\end{enumerate}

As $F(u)=O(|u|)$ we derive the expression:
$$
\tr(\hol_\gamma F(u))=2+\frac{1}{4}\tr\left(\int_\gamma F(u)\right)^2+\frac{1}{24}\tr\int_{\Delta^4}\gamma^* F(u)+O(|u|^5).$$
 As $F_1(u)$ is cohomologous to $u$ we have $\int_\gamma F_1(u)=\langle\gamma,u\rangle$. This proves the first part of the proposition. 
Suppose now that $\gamma$ separates and denote by $\Sigma'$ the submanifold of $\Sigma$ with $\partial \Sigma'=\gamma$. 

We have $\int_\gamma F_1(u)=0$ and $\int_\gamma F_2(u)=\int_{\Sigma'}dF_2(u)=-\frac{1}{2}\int_{\Sigma'}[F_1(u)\wedge F_1(u)]$. 
This shows the formula
$$\tr(\hol_\gamma F(u))=2+\frac{1}{16}\tr\left(\int_{\Sigma'} [F_1(u)\wedge F_1(u)]\right)^2+\frac{1}{24}\tr\int_{\Delta^4}\gamma^*F_1(u)^4+O(|u|^5).$$

To conclude we observe that the left hand side is invariant by a gauge transformation, that is we can replace $F(u)$ by $gF(u)g^{-1}-(\d g) g^{-1}$ for some $g:\Sigma\to G$. Writing $g=\exp(\xi)$ for some $\xi:\Sigma\to \g$, we compute that an infinitesimal gauge transformation maps $F_1(u)$ to $F_1(u)-\d \xi$. 
We conclude that the right hand side is invariant by such transformations and hence depends only on the cohomology class $u$ of $F_1(u)$. 
Taking a 1-form $\alpha'_1$ cohomologous to $u$ and such that $\gamma^*\alpha'_1=0$, we get the result of the proposition. 
\end{proof} 

\section{Non-separating trace functions around the trivial representation}

For any oriented separating curve $\gamma\subset \Sigma$, denote by $g_\gamma: H^1(\Sigma,\g)\to \R$ the map defined by 
$$g_\gamma(u)=\tr[u'\wedge u']^2$$
 where $\Sigma'$ is the subsurface of $\Sigma$ so that $\partial \Sigma'=\gamma$ and $u=u'+u''$ is the decomposition of $u$ in $H^1(\Sigma,\g)=H^1(\Sigma',\g)\oplus H^1(\Sigma'',\g)$. 

Identifying $H^1(\Sigma,\g)$ to $\hom(H_1(\Sigma,\R),\g)$ we define $C_{\rm irr}\subset C$ as the subset of surjective elements. 
\begin{proposition}\label{engendrecone}
For any surface $\Sigma$ of genus $g>1$ there is a finite set of separating curves $S$ of cardinality $g(2g+1)(2g^2+g+1)/2$ and a neighborhood $U$ of $0$ in $H^1(\Sigma,\g)$ such for all $u$ in $C_{\rm irr}\cap U$ and $v\in H^1(\Sigma,\g)$ we have:
$$ [v,u]=0\text{ and }D_u g_\gamma(v)=0,\forall \gamma\in S
\quad\Rightarrow\quad
 v=[\xi,u]\textrm{ for some }\xi\in\g$$
\end{proposition}
This amounts to say that there is a neighborhood $V$ of 0 in $C/G$ such that the derivatives of the functions $g_{\gamma}$ for $\gamma$ in $S$ generate the cotangent space at every surjective $u\in V$.  
The reason is that for any $u\in C_{\rm irr}$ we have:
$$T_u (C_{\rm irr}/G)=\{v\in H^1(\Sigma,\g)\text{ such that  }[u\wedge v]=0\}/\{[\xi\wedge u]\text{ for }\xi\in H^0(\Sigma,\g)\}.$$

\begin{proof}
Let $\omega$ denotes the intersection product on $H_1(\Sigma,\R)$. Let $\gamma$ be a separating curve, $\Sigma',\Sigma''$ the corresponding subsurfaces and write $u=P(u)+(u-P(u))$ the decomposition of any $u$ in $H^1(\Sigma,\g)=H^1(\Sigma',\g)\oplus H^1(\Sigma'',\g)$. 
  
A direct computation shows:
 $$D_u g_\gamma(v)=4\tr([P(u)\wedge P(u)][P(u)\wedge P(v)])$$
Suppose that this quantity vanishes for all curves $\gamma$ bounding a 
1-holed torus $\Sigma'$. 
It implies that for all rank 2 lattices $\Lambda\subset H_1(\Sigma,\Z)$ 
such that $\Lambda\oplus \Lambda^{\bot_\omega}=H_1(\Sigma,\Z)$ 
(where $\Lambda^{\bot_\omega}$ is the symplectic orthogonal of $\Lambda$) 
we have
$$\tr([P_{\Lambda}(u)\wedge P_{\Lambda}(u)][P_{\Lambda}(u)\wedge P_{\Lambda}(v)])=0$$
where $P_\Lambda$ is the projection on $\hom(\Lambda,\g)$ 
parallel to $\hom(\Lambda^{\bot_\omega},\g)$.

Take $(x,y)$ a symplectic base of $\Lambda$ and denote by 
$p_{\Lambda}$ the symplectic projection of $H_1(\Sigma,\R)$ onto 
$\Lambda\otimes\R$. 
Identifying $u$ and $v$ to elements in $\hom(H_1(\Sigma,\R),\g)$, $P_\Lambda(u)$ and $P_\Lambda(v)$ are identified with $u\circ p_\Lambda$ and $v\circ p_\Lambda$ respectively.

Moreover, $[u\wedge v]=c\circ (u\otimes v) \circ \omega^{-1}$ where $\omega^{-1}\in H_1(\Sigma,\R)^{\otimes 2}$ represents the inverse symplectic product and $c:\g\otimes \g\to \g$ stands for the Lie bracket. 
We deduce from this formula the identity:
$$[P_\Lambda(u)\wedge P_\Lambda(v)]=[u(x),v(y)]-[u(y),v(x)]$$

Let $Q_{u,v}:H^1(\Sigma,\R)\times H^1(\Sigma,\R)\to\R$ be the map defined by 
$$Q_{u,v}(x,y)=\tr([u(x),u(y)]\left([u(x),v(y)]-[u(y),v(x)]\right))$$

We use the following lemma in order to reduce the statement to prove the vanishing of $Q_{u,v}$. We postpone its proof to the end of the section as we could not find a proof of it avoiding computations. 
\begin{lemma}\label{calcul}
Let $u,v\in H^1(\Sigma,\mathfrak g)$ such that 
$u$ is surjective, $[u\wedge u]=[u\wedge v]=0$ and $Q_{u,v}= 0$.
Then there exists $\xi\in \g$ so that $v=[\xi\wedge u]$.
\end{lemma}

It remains to show that the vanishing of $Q_{u,v}$ can be detected 
by a finite number of rank 2 symplectic lattices in 
$H_1(\Sigma,\Z)$. This is the content of the following lemma.

Let  $S=\{x_i\otimes y_i, i\in I\}\subset S^2H_1(\Sigma,\Z)$ be a set 
of vectors satisfying $\omega(x_i,y_i)=1$.  
We say that $S$ is quadratically generating if any quadratic form on $S^2 H_1(\Sigma,\R)$ vanishing on $S$ vanishes everywhere. 

\begin{lemma}\label{quadraticgeneration}
There exists quadratically generating sets $S$. Moreover one can find such sets with cardinality $g(2g+1)(2g^2+g+1)/2$.
\end{lemma}
Again we postpone the proof this lemma to the end of the section. Observe that the proposition follows by considering any quadratically generating set $S$. By the assumption of the proposition we have for any $u,v$ and any $x_i\otimes y_i\in S$ the equality $Q_{u,v}(x_i,y_i)=0$. The generating property implies that $Q_{u,v}=0$ and Lemma \ref{calcul} implies the result.

\end{proof}

\begin{proof}[Proof of Lemma \ref{calcul}:]
Let $(\xi_i)_{i\in \Z_3}$ be a basis of $\g$ so that $[\xi_i,\xi_{i+1}]=\xi_{i+2}$ and normalize the trace so that $\tr(\xi_i\xi_j)=\delta_{ij}$ for all $i,j\in\Z_3$.  

Write $u=\sum_i u_i\xi_i$ for $u_i\in H^1(\Sigma,\R)$. The hypothesis on $u$ imply that the $u_i$ are non zero and mutually orthogonal. Fix a basis $(e_i)_{0\le i< n}$ of $H_1(\Sigma,\R)$ so that $u_i=e_i^*$ for $i=0,1,2$. 
Then $v=\sum_{i\in\Z_3}\sum_{j<n}v_i^l e_l^*\otimes \xi_i$, $x=\sum_l x_l e_l$ and $y=\sum_l y_l e_l$. 

From now on, $i\in \Z_3$ and $0\le l<n$. We have $u(x)=\sum_i x_i \xi_i$ and $v(x)=\sum_{i,l} v_i^lx_l\xi_i$ so that:

$$[u(x),u(y)]=\sum_i(x_iy_{i+1}-x_{i+1}y_i)\xi_{i+2}$$
$$[u(x),v(y)]=\sum_iy_l(x_iv_{i+1}^l-x_{i+1}v_i^l)\xi_{i+2}$$
This gives the identity
$$\sum_{i,l}\left((x_iy_{i+1}-x_{i+1}y_i)(x_iy_lv_{i+1}^l-x_{i+1}y_lv_i^l-x_ly_iv_{i+1}^l+x_ly_{i+1}v_i^l)\right)=0$$

This biquadratic polynomial in $x$ and $y$ vanishes identically if and only if all its coefficient vanish.
\begin{itemize}
\item[-]The coefficient of $x_i^2y_{i+1}y_l$ for $l>3$ is $v_{i+1}^l$ showing that $v_i^l=0$ for all $i$ and $l>3$. 
\item[-]The coefficient of $x_i^2y_{i+1}^2$ is $v_{i+1}^{i+1}+v_{i}^{i}$ showing that $v_i^i=0$ for all $i$. 
\item[-]The coefficient of $x_i^2y_{i+1}y_{i-1}$ is $v_{i+1}^{i-1}+v_{i-1}^{i+1}$.
\end{itemize}
This show that the matrix $v_i^j$ is antisymmetric so that the set of solutions have dimension 3 as the set of infinitesimal symmetries. 
\end{proof}

\begin{proof}[Proof of Lemma \ref{quadraticgeneration}:]
The vector space $\mathcal{Q}$ of quadratic forms on $S^2H^1(\Sigma,\R)$ has dimension $N=g(2g+1)(2g^2+g+1)/2$.   
Each  symplectic lattice $\Lambda_i=\span(x_i,y_i)$ 
acts linearly on $\mathcal{Q}$ by evaluation 
at $x_i\otimes y_i$. If these linear forms generate the dual of 
$\mathcal{Q}$,  then a subset of size $N$ of them will do. 

It remains to prove that these linear forms indeed generate the dual 
of $\mathcal{Q}$.  Let then $Q\in  \mathcal{Q}$ such that, 
for any $x,y\in H_1(\Sigma,\Z)$ satisfying 
$\omega(x,y)=\pm 1$ we have $Q(x,y)=0$. 

Let $(x_0,y_0)$ be such an integral symplectic basis. Then 
the map $\sp(2g,\R)\to \R$ sending $A$ to $Q(Ax_0,Ay_0)$ 
is algebraic. This map  vanishes on $\sp(2g,\Z)$, which is 
Zariski dense in $\sp(2g,\R)$ and hence, it vanishes identically. 
As the group $\sp(2g,\R)$ acts transitively on  the set of 
pairs $(x,y)\in H_1(\Sigma,\R)$ satisfying $\omega(x,y)=1$,  
we get $Q(x,y)=0$ for such pairs. 
As $Q$ is homogeneous, we have the same result for  all $(x,y)$ such that 
$\omega(x,y)\ne 0$. By density, we get $Q=0$.
\end{proof}

Let $S$ be a set of simple curves $\gamma_i$ 
bounding 1-holed tori in $\Sigma$. We will say that this set is quadratically generating if the corresponding  2-dimensional symplectic lattices are quadratically generating.

\begin{proposition}\label{engendremodule}
Let $S$ be a quadratically generating set of curves, and $f_\gamma:\mo(\Sigma,G)\to \R$ be 
the associated trace functions. 

For any conical neighborhood $V$ of the subset of reducible representations in $\mo(\Sigma,G)$ there is a neighborhood $U$ of the trivial representation 
in $\mo(\Sigma,G)$ so that for any $x\in U\setminus V$ 
the derivatives of $f_\gamma$ at $x$, with $\gamma\in S$, 
generate the cotangent space $T_x^*\mo(\Sigma,G)$. 
\end{proposition}

\begin{remark} 
Observe that $\mo(\Sigma,G)$ is locally modeled on the cone $C/G$. By conical neighborhood $V$ we mean that $V$ contains some neighborhood invariant by scaling in the model. This is indeed independent on the real analytic diffeomorphism we choose. 
\end{remark}
   
\begin{proof}
The proposition is equivalent to saying that for 
$\rho\in\rep(\Sigma,G)$ close enough to the trivial representation but far enough from reducible representations, 
the  intersection of the kernels of  the derivatives of the functions 
$f_\gamma$, with $\gamma\in S$, is precisely  the tangent 
space of the $G$ action.  
Recall that the map $F:C\to \rep(\Sigma,G)$ is a $G$-equivariant real analytic diffeomorphism, hence it is sufficient to prove the analogous claim 
for $f_\gamma\circ F$.  
From Proposition \ref{taylor}, we have at each point $u$ of $C$:
$$f_\gamma(F(u))=2 + \frac{1}{16}g_\gamma(u)+O(|u|^5).$$
As $F$ is a smooth function of $u$, we have 
\begin{equation}\label{dl}
16D_u (f_\gamma\circ F)=D_u g_\gamma+O(|u|^4).
\end{equation}
By Proposition \ref{engendrecone}, there is a neighborhood 
$U$ of $0$ in $C$ such that for all $u$ in $C_{\rm irr}\cap U$, the 
derivatives of $g_\gamma$, $\gamma\in S$, 
generate the cotangent space of $C$ at $u$, modulo the 
3-dimensional space of elements of the form $[\xi,u]$, with 
$\xi\in H^0(\Sigma,\mathfrak g)$. 
This means that we can find a set $S'\subset S$ of cardinal $6g-6$,  
such that the derivatives of $g_\gamma$, $\gamma\in S'$,  
are linearly independent at $u$. This is an open condition and thus 
it holds true for some open neighborhood $U_{S'}$ containing $u$.
In particular $\{U_{S'}\}$, for $S'\subset S$ (of cardinal $6g-6$) 
forms an open covering of $C_{\rm irr}\cap U$.  
Taking an auxiliary euclidean structure in $H^1(\Sigma,\g)$, 
this last condition can be expressed by saying that the Gram 
determinant $\textrm{Gram}(D_u g_\gamma)_{\gamma\in S'}$ of 
$(D_u g_\gamma)_{\gamma\in S'}$ relative to the 
euclidean structure is positive on $U_{S'}$. Thus 
$\Phi(u)=\sum_{S'\subset S} \textrm{Gram}(D_u g_\gamma)_{\gamma\in S'}$ is 
positive on $C_{\rm irr}\cap U$.

Define $\Psi(u)=\sum_{S'\subset S} \textrm{Gram}(16 D_u f_\gamma\circ F)_{\gamma\in S'}$. By reversing the argument above, we need to prove that there is a neighborhood $U'$ of $0$ in $C$ so that for any $u\in U'\setminus V$ one has $\Psi(u)>0$. 

We observe that $D_ug_\gamma$ is cubic in $u$ and so $\Phi(u)$ is homogeneous of degree $N=36(g-1)$. By using Equation \eqref{dl} and expanding the Gram determinant, we find that 
\begin{equation}\label{grando}
\Psi(u)=\Phi(u)+O(|u|^{N+1}).
\end{equation}

Let $V$ be the conic neighborhood of $C\setminus C_{\rm irr}$ given in the statement and $S$ be the compact set defined by $S=\{u\in C, |u|=1\}$. As $S\setminus V$ is a compact subset of $C_{\rm irr}$, there exists $\epsilon>0$ so that $\Phi>\epsilon$ on $S\setminus V$. By homogeneity we get $|\Phi(u)| >\epsilon|u|^N$ on $C_{\rm irr}\setminus V$ which together with Equation \eqref{grando} proves the result. 

\end{proof}
\begin{remark}\label{connexite} We will call an open set of the form $U\setminus \overline{V}$ given by the above proposition a "neighborhood of the trivial representation". We can use these neighborhoods to define a topology on $\mo_{\rm irr}(\Sigma,G)\cup\{1\}$. A crucial observation is that these open sets are connected, as $C_{\rm irr}$ is locally connected around the trivial 
representation. In fact, according to Lemma \ref{calcul} 
a point of $C_{\rm irr}$ is given by a triple of pairwise orthogonal 
elements of $H^1(\Sigma; \R)$ and the symplectic group ${\rm Sp}(2g,\R)$ acts 
transitively on the space of such triples.  
\end{remark}

\section{Ergodicity of the Johnson subgroup action on $\mo(\Sigma,\su)$}

Fix a set of separating curves $S$ of genus one which are quadratically generating. 
Let $\mathcal K_1(\Sigma;S)\subset \mathcal K(\Sigma)$ 
denote the normal subgroup of $\M(\Sigma)$ generated by the Dehn twists 
along the curves $\gamma\in S$. Further define, for all $n\geq 1$,  
$\mathcal K_{n+1}(\Sigma; S)$ as being the normal subgroup in 
$\mathcal K_n(\Sigma; S)$ generated by the Dehn twists along curves 
in $S$. The result of Theorem \ref{main} follows from the more general: 

\begin{proposition}\label{bootstrap}
For each $n\geq 1$ the action of $\mathcal K_n(\Sigma, S)$ 
on $\mo(\Sigma,\su)$ is ergodic. 
\end{proposition}
\begin{proof}
For $n\geq 1$ we denote by   
$\mathcal A_n(S)$ the set of orbits of curves in $S$ under 
the action of $\mathcal K_{n-1}(\Sigma;S)$, where $\mathcal K_0(\Sigma;S)$ 
states for $\M(\Sigma)$.   Define the set: 
$$U_n=\{\rho \in \mo_{\rm irr}(\Sigma,G)\text{ such that }\span\{D_\rho f_\gamma\text{ for }\gamma\in \mathcal A_n(S)\}=T^*_\rho\mo(\Sigma,G)\}.$$

As this is an open condition $U_n$ is open. Moreover 
$U_1$ is invariant by the mapping class group. 

We will prove the claim by recurrence on $n$ by means of a 
bootstrap argument. Consider first $n=1$. 
Remark \ref{connexite} directly implies that 
there is a connected neighborhood $V$ of the trivial representation  
in $\mo(\Sigma,G)$ so that 
$V\cap\mo_{\rm irr}(\Sigma,G)\subset U_1$. 
Let $U^0_1$ be the unique connected component of $U_1$ which contains 
$V$. The open set $U^0_1$ is then non empty and 
invariant by the mapping class group. By the ergodicity of the mapping 
class group on $\mo(\Sigma, \su)$, proved by Goldman in \cite{Go2}, 
the  complement of  $U^0_1$ has measure 0. 
The end of the proof for $n=1$ follows from the 
arguments of Goldman and Xia in \cite{GoX}. 
Specifically, one key ingredient is the infinitesimal transitivity 
Lemma 3.2 from \cite{GoX}, which we state 
here for the sake of completeness: 

\begin{lemma}\label{transitive}
Let $X$ be a connected symplectic manifold and $\mathcal F$ be a set of functions 
such that their differential at all points of $X$ span the cotangent space.
Then the group generated by the Hamiltonian flows associated to the functions 
in $\mathcal F$  acts transitively on $X$.  
\end{lemma}

The Hamiltonian flow associated to the trace 
function  $f_{\gamma}$ 
is covered by the hamiltonian flow $\Phi_\gamma^t$ of $h_\gamma=\arccos(f_{\gamma}/2)$, the so-called Goldman twist flow 
(see e.g. \cite{Go2,GoX}) defined by $\gamma$. The flow $\Phi_\gamma$ gives a circle action of period $\pi$ on an open dense subset of 
$\mo(\Sigma, \su)$. The action of the Dehn twist
along the separating simple curve $\gamma$ on $\mo(\Sigma, \su)$ is identified 
with  the time $h_{\gamma}$ of 
the Goldman twist flow. 

This circle action is a rotation of angle 
$h_\gamma$ and therefore the 
Dehn twist along $\gamma$ acts ergodically on the orbit of $\rho\in \mo(\Sigma, \su)$ 
under the Goldman twist flow defined by $\gamma$, for all $\rho$ with irrational 
$\frac{h_\gamma(\rho)}{\pi}$.  
Now, this implies (see \cite[Proposition 5.4]{GoX}) that any measurable function on $\mo(\Sigma, \su)$ 
which is invariant by the action of the Dehn twist along $\gamma$ 
should be constant on the orbits of the Goldman twist flow defined by $\gamma$ 
outside a nullset of $\mo(\Sigma, \su)$. Therefore, 
any measurable function on $\mo(\Sigma, \su)$ which is invariant by the 
group $\mathcal K_1(\Sigma;S)$
should be invariant by the group generated 
by the Hamiltonian flows associated to the functions $f_{\gamma}$, 
for $\gamma\in S$, almost everywhere.   
Then, by the transitivity Lemma \ref{transitive} 
and Proposition \ref{engendremodule} it must be constant on 
$U_1^0$ almost everywhere. 
This establishes the claim for $n=1$.

Assume now that the claim holds for $n$. 
One can find  a connected neighborhood 
$V_n$ of the trivial representation  in $\mo(\Sigma,G)$ so that 
$V_n\subset U_n$ using again Remark \ref{connexite}. 
Let $U^0_n$ be the unique connected component of $U_n$ which contains $V_n$. 
The open set $U^0_n$ is then non empty and 
invariant by the group $\mathcal K_n(\Sigma; S)$. 
By using the ergodicity of the  $\mathcal K_n(\Sigma; S)$ action 
on $\mo(\Sigma, \su)$, which is the induction hypothesis, 
the complement of $U^0_n$ has measure 0. 
Then again the arguments from \cite{GoX} imply that 
the $\mathcal K_{n+1}(\Sigma; S)$ action is ergodic.  
\end{proof}

\begin{remark}
The subgroup $\Gamma(S)$  generated by the Dehn twists 
along curves in $S$ is contained in the intersection
$\cap_{n\geq 1} K_n(\Sigma;S)$, but we don't know whether 
the inclusion is strict. One also might speculate that
$\Gamma(S)$ acts ergodically on $\mo(\Sigma,\su)$.  
\end{remark}

\bibliographystyle{plain}

\end{document}